\documentclass[11pt]{amsart}
\usepackage{amsmath}
\usepackage{amsthm}
\usepackage{amssymb}
\usepackage{verbatim}
\usepackage{color}

\newtheorem{theorem}{Theorem}    

\newtheorem{lemma}[theorem]{Lemma}

\newtheorem{remark}[theorem]{Remark}
\newtheorem{definition}[theorem]{Definition}
\theoremstyle{definition}

\numberwithin{theorem}{section} \numberwithin{theorem}{section}
\numberwithin{equation}{section}

\allowdisplaybreaks

\begin{document}
\title[A note on fractional type integrals in the Schr\"{o}dinger setting]
{A note on fractional type integrals in the Schr\"{o}dinger setting}

\author{Yongming Wen}


%
\thanks{Supported by the Natural Science Foundation of Fujian Province(Nos. 2021J05188), President's fund of Minnan Normal University (No. KJ2020020), Institute of Meteorological Big Data-Digital Fujian, Fujian Key Laboratory of Data Science and Statistics and Fujian Key Laboratory of Granular Computing and Applications (Minnan Normal University), China.}
\address{School of Mathematics and Statistics, Minnan Normal University, Zhangzhou 363000,  China} \email{wenyongmingxmu@163.com}



\begin{abstract}
Assume $\mathcal{L}=-\Delta+V$ is a Schr\"{o}dinger operator on $\mathbb{R}^d$, where $V$ belongs to certain reverse H\"{o}lder class $RH_\sigma$ with $\sigma\geq d/2$. We consider the class of $A_{p,q}$ weights associated to $\mathcal{L}$, denoted by $A_{p,q}^{\mathcal{L}}(\mathbb{R}^d)$, which include the classical Muckenhoupt $A_{p,q}(\mathbb{R}^d)$ weights. We obtain the quantitative $A_{p,q}^{\mathcal{L}}(\mathbb{R}^d)$ estimates for fractional integrals associated to the Schr\"{o}dinger operator. Particularly, the quantitative weighted endpoint bound for fractional integrals associated to the Schr\"{o}dinger operator is first established, which was missing in the literature of Li et al. \cite{LRW}. Moreover, we generalize weighted endpoint inequalities to weighted mixed weak type inequalities for fractional type integrals in the Schr\"{o}dinger setting.
\end{abstract}

\maketitle

\section{Introduction and main results}
\subsection{Background}

In recent years, the problem of quantitative weighted estimates for operators in Laplacian
setting have appealed to many mathematician. The initial result was opened by Buckley \cite{Buckley}, who proved the following sharp weighted estimates for Hardy-Littlewood maximal operator $M$
\begin{align}\label{HL}
\|Mf\|_{L^p(\omega)}\lesssim[\omega]_{A_p(\mathbb{R}^d)}^{1/(p-1)}\|f\|_{L^p(\omega)},~1<p<\infty.
\end{align}
To resolve an important endpoint result in the theory of quasiconformal mappings that had been conjecture by Astala, Iwaniec and Saksman \cite{AIS}, Petermichl and Volberg \cite{PV} settled the quantitative weighted estimates for Beurling transform. Petermichl \cite{P,P1} also obtained the sharp bounds for the Hilbert and Riesz transforms. While for general Calder\'{o}n-Zygmund operators, Hyt\"{o}nen \cite{Hy} proved the dyadic representation theorem for Calder\'{o}n-Zygmund operators, this leads to the proof of $A_2$ conjecture. In 2010, Lacey, Moen, P\'{e}rez and Torres \cite{LaMPT} obtained sharp bounds for the classical fractional integral operators
$$I_\alpha f(x)=\int_{\mathbb{R}^d}\frac{f(y)}{|x-y|^{n-\alpha}}dy$$
as follows.
\begin{theorem}{\rm(cf. \cite{LaMPT})}\label{classical fractional}
Suppose that $0<\alpha<d$, $1\leq p<d/\alpha$ and that $q$ satisfies $1/q=1/p-\alpha/d$. Then
\begin{align*}
\|I_\alpha f\|_{L^q(\omega^q)}\lesssim[\omega]_{A_{p,q}(\mathbb{R}^d)}^{(1-\frac{\alpha}{d})\max\{1,p'/q\}}
\|f\|_{L^p(\omega^p)},~1<p<d/\alpha,
\end{align*}
and
\begin{align*}
\|I_\alpha f\|_{L^{q,\infty}(\omega^q)}\lesssim[\omega]_{A_{p,q}(\mathbb{R}^d)}^{1-\frac{\alpha}{d}}
\|f\|_{L^p(\omega^p)},~
1\leq p<d/\alpha.
\end{align*}
Furthermore, both of two results above are sharp.
\end{theorem}
Very recently, under the umbrella of ``domination by sparse operators'', which give a new way to obtain the quantitative weighted estimates for harmonic analysis operators, there is a vast literature concerned with this topic, see \cite{CaIRXY,CaXYa,CaYa,ChLiHo,CoCuDiOu,DoLaR,HyLi,LaLi,Ler,Ler1,Li,LiPeRiRo} et al., among these, the work of Lerner \cite{Ler,Ler1} played a center role. Quantitative weighted bounds for operators in the Schr\"{o}dinger setting began by the work of Li, Rahm and Wick \cite{LRW}. To state their results, we first recall some necessary definitions. We consider the Schr\"{o}dinger operator $\mathcal{L}$ on $\mathbb{R}^d$ with $d\geq3$,
$$\mathcal{L}=-\Delta+V,$$
where $\Delta$ is the standard Laplacian operator in $\mathbb{R}^d$ and $V$ is a non-negative potential which belongs to certain reverse H\"{o}lder class $RH_\sigma$ ($\sigma>d/2$), that is,
$$\Big(\frac{1}{|B|}\int_BV(y)^\sigma dy\Big)^{1/\sigma}\leq\frac{C}{|B|}\int_BV(y)dy,$$
for every ball $B\subset \mathbb{R}^d$. Define fractional integral $\mathcal{L}^{-\alpha/2}$ and fractional maximal function $M_\alpha^{\rho,\theta}$ as below.
\begin{align*}
\mathcal{L}^{-\alpha/2}f(x)=\int_0^\infty e^{-t\mathcal{L}}f(x)t^{\frac{\alpha}{2}-1}dt,~0<\alpha<d,
\end{align*}
\begin{align*}
M_\alpha^{\rho,\theta}f(x)=\sup_{Q\ni x}\frac{1}{\psi_\theta(Q)|Q|^{1-\frac{\alpha}{d}}}\int_Q|f(y)|dy,~\theta\geq0,
\end{align*}
where $\psi_\theta(Q):=(1+r_Q/\rho(x_Q))^\theta$, $\rho$ is the critical radius function (see Section 2 for a precise definition), and $x_Q$, $r_Q$ are the center of cube $Q$ and the side-length of $Q$, respectively. In \cite{LRW}, the authors obtained following quantitative weighted estimates for $\mathcal{L}^{-\alpha/2}$ and $M_\alpha^{\rho,\theta}$.
\begin{theorem}{\rm(cf. \cite{LRW})}\label{schrodinger fractional strong type}
Suppose that $0<\alpha<d$. Let $1\leq p<d/\alpha$ and that $q$ satisfies $1/q=1/p-\alpha/d$. Let $\theta\geq0$, $\gamma=\theta/\big(1+\frac{p'}{q}\big)$ and $K$ be defined by the equation $\big(\frac{1}{K}+\frac{q}{Kp'}\big)\big(1-\frac{\alpha}{d}\big)\max\{1,p'/q\}=1/2$. Then
\begin{align*}
\|\mathcal{L}^{-\alpha/2} f\|_{L^q(\omega^q)}\lesssim[\omega]_{A_{p,q}^{\rho,\frac{\theta}{3K}}(\mathbb{R}^d)}
^{(1-\frac{\alpha}{d})\max\{1,p'/q\}}
\|f\|_{L^p(\omega^p)},~\omega\in A_{p,q}^{\rho,\frac{\theta}{3K}}(\mathbb{R}^d),
\end{align*}
and
\begin{align*}
\|M_\alpha^{\rho,\theta} f\|_{L^{q}(\omega^q)}\lesssim[\omega]_{A_{p,q}^{\rho,\frac{\gamma}{3}}(\mathbb{R}^d)}^
{p'(1-\frac{\alpha}{d})/q}
\|f\|_{L^p(\omega^p)},~\omega\in A_{p,q}^{\rho,\frac{\gamma}{3}}(\mathbb{R}^d).
\end{align*}
\end{theorem}
Subsequently, Zhang and Yang \cite{ZhYa} obtained quantitative weighted strong $(p,p)$ type estimates for Littlewood-Paley operators associated to $\mathcal{L}$. The author and Wu \cite{WW} recently investigated quantitative endpoint estimates for maximal operators and variation operators associated to $\mathcal{L}$. More generally, Bui et al. \cite{BuiBD,BuiBD1} achieved quantitative weighted strong $(p,p)$ type estimates for square functions and singular integrals associated to general differential operators.

On the other hand, to seek a new proof for the qualitative estimate of \eqref{HL}, in the case of dimension 1, Sawyer \cite{S} showed that it suffices to establish the following result: if $\mu,\nu\in A_1(\mathbb{R})$, then
\begin{align}\label{HL mix}
\Big\|\frac{M(f\nu)}{\nu}\Big\|_{L^{1,\infty}(\mu\nu)}\lesssim\|f\|_{L^1(\mu\nu)}.
\end{align}
It is obvious that \eqref{HL mix} reduces to the weighted weak $(1,1)$ type of $M$ when $\nu=1$. However, the proof of \eqref{HL mix} is highly non-trivial due to the covering lemmas do not apply for $M(f\nu)/\nu$ and $\mu\nu$ may be very singular. We refer these types of estimates as weighted mixed weak type estimates. Later, Cruz-Uribe, Martell and P\'{e}rez \cite{CrMP} extended \eqref{HL mix} to higher dimension and Calder\'{o}n-Zygmund with weights $\mu,\nu\in A_1(\mathbb{R}^d)$ or $\mu\in A_1(\mathbb{R}^d)$ and $\nu\in A_\infty(\mu)$. Li, Ombrosi and P\'{e}rez \cite{LiOP} improved the results in \cite{CrMP} by assuming $\mu\in A_1(\mathbb{R}^d)$, $\nu\in A_\infty(\mathbb{R}^d)$. In the case of fractional type integrals, Berra, Carena and Pradolini \cite{BeCaPr} proved the following weighted mixed weak type inequalities.
\begin{theorem}{\rm(cf. \cite{BeCaPr})}\label{mix fractional}
Let $0<\alpha<d$, $1\leq p<d/\alpha$ and $q$ satisfy $1/q=1/p-\alpha/d$. If $\mu,\nu$ are weights such that $\mu,\nu^{q/p}\in A_1(\mathbb{R}^d)$ or $\mu\nu^{-q/p'}\in A_1(\mathbb{R}^d)$ and $\nu^q\in A_\infty(\mu\nu^{-q/p'})$, then there exists a positive constant $C$ such that for every $t>0$ and every $f\in L_c^\infty(\mathbb{R}^d)$
\begin{align*}
\mu\nu^{q/p}\Big(\Big\{x\in\mathbb{R}^d:\frac{T(f\nu)(x)}{\nu(x)}>t\Big\}\Big)^{1/q}\leq\frac{C}{t}
\Big(\int_{\mathbb{R}^d}|f(x)|^p\mu(x)^{p/q}\nu(x)dx\Big)^{1/p},
\end{align*}
where $T$ is the fractional maximal operator or $I_\alpha$.
\end{theorem}
Recently, Berra, Pradolini and Quijano \cite{BePrQu} first established weighted mixed weak type inequalities for Hardy-Littlewood maximal operator and singular integrals in the Schr\"{o}dinger setting. For more works on this topic, we refer readers to see \cite{CalRi,CaIRXY,CaXYa,LiOB,LiPeRiRo,IR} etc.

\subsection{Aims and questions}
The aim of this paper is to continue the line of Li-Rahm-Wick \cite{LRW} and Berra-Pradolini-Quijano \cite{BePrQu} to study the quantitative weighted estimates and weighted mixed weak type inequalities in the Schr\"{o}dinger setting. There are several problems of fractional type integrals associated to $\mathcal{L}$ remain to resolve.

In \cite{LRW}, the authors obtained quantitative weighted strong $(p,q)$ type estimates for $\mathcal{L}^{-\alpha/2}$. It is very natural to ask the following question.\\
\textbf{Question 1:} Can we establish quantitative weighted endpoint estimates for $\mathcal{L}^{-\alpha/2}$?

Berra et al. \cite{BeCaPr} obtained weighted mixed weak type inequalities for fractional maximal operator and fractional integral operators associated to $\Delta$. The proof of Theorem \ref{mix fractional} relies heavily on extrapolation theorem established in \cite{CrMP} and Coifman type inequality proved in \cite{MW}. However, the classes of weights associated to $\mathcal{L}$ is larger than classes of classical Muckenhoupt weights. Besides, deficiency of the regularity that $-\Delta$ possesses, operator of the form $\mathcal{L}=-\Delta+V$ present many challenges. Hence, the technique used in \cite{BeCaPr} may not be applied to fractional integral operators.\\
\textbf{Question 2:} How to achieve the weighted mixed weak type inequalities for $M_\alpha^{\rho,\theta}$ and $\mathcal{L}^{-\alpha/2}$?

\subsection{Main results}
Our first result is concerned with the quantitative weighted estimates for $\mathcal{L}^{-\alpha/2}$.
\begin{theorem}\label{schrodinger fractional weak type}
Let $0\leq\alpha<d$, $1\leq p<d/\alpha$, $1/q=1/p-\alpha/d$ and $\rho$ be a critical radius function. Let $\gamma=\theta/(1+p'/q)$ with $\theta\geq0$.
\medskip

{\rm (i)}
For $\omega\in A_{p,q}^{\rho,\gamma/3}(\mathbb{R}^d)$,
\begin{align*}
\|\mathcal{L}^{-\alpha/2}f\|_{L^q(\omega^q)}\lesssim[\omega]_{A_{p,q}^{\rho,\gamma/3}(\mathbb{R}^d)}
^{(1-\frac{\alpha}{d})\max\{1,\frac{p'}{q}\}}\|f\|_{L^p(\omega^p)}, ~1<p<d/\alpha;
\end{align*}
\medskip

{\rm (ii)}For $\omega\in A_{p,q}^{\rho,\theta}(\mathbb{R}^d)$,
\begin{align*}
\|\mathcal{L}^{-\alpha/2}f\|_{L^{q,\infty}(\omega^q)}\lesssim
[\omega]_{A_{p,q}^{\rho,\theta}(\mathbb{R}^d)}^{1-\alpha/d}
\|f\|_{L^p(\omega^p)}.
\end{align*}
\end{theorem}
\begin{remark}
This is a full version of Theorem \ref{classical fractional} adapted to the Schr\"{o}dinger setting. Besides, we provide a different method to the quantitative weighted strong $(p,q)$ type estimates for $\mathcal{L}^{-\alpha/2}$. However, this result is not comparable to the one in Theorem \ref{schrodinger fractional strong type}. For instance, if $p'>q$ and $0\leq\alpha<d/2$, our result is better, while the result in Theorem \ref{schrodinger fractional strong type} is better provided that $p'>q$ and $d/2<\alpha<d$.
\end{remark}

Our next theorems give a positive answer to the second question.
\begin{theorem}\label{mix schrodinger fractional maximal}
Let $0\leq\alpha<d$, $1\leq p<d/\alpha$, $1/q=1/p-\alpha/d$ and $\rho$ be a critical function. If $\mu\nu^{-q/p'}\in A_1^\rho(\mathbb{R}^d)$ and $\nu^q\in A_{\infty}^{\rho}(\mu\nu^{-q/p'})$. Then for any $t>0$ and $f\in L_c^\infty(\mathbb{R}^d)$, there exists $\theta\geq0$ such that
\begin{align*}
\mu\nu^{q/p}\Big(\Big\{x\in\mathbb{R}^d:\frac{M_\alpha^{\rho,\theta}(f\nu)(x)}{\nu(x)}>t\Big\}\Big)
^{1/q}\lesssim
\frac{1}{t}\Big(\int_{\mathbb{R}^d}|f(x)|^p\mu(x)^{p/q}\nu(x)dx\Big)^{1/p}.
\end{align*}
\end{theorem}

\begin{theorem}\label{mix schrodinger fractional}
Let $0\leq\alpha<d$, $1\leq p<d/\alpha$, $1/q=1/p-\alpha/d$ and $\rho$ be a critical function. If $\mu\nu^{-q/p'}\in A_1^\rho(\mathbb{R}^d)$ and $\nu^q\in A_{\infty}^{\rho}(\mu\nu^{-q/p'})$. Then for any $t>0$ and $f\in L_c^\infty(\mathbb{R}^d)$,
\begin{align*}
\mu\nu^{q/p}\Big(\Big\{x\in\mathbb{R}^d:\frac{\mathcal{L}^{-\alpha/2}(f\nu)(x)}{\nu(x)}>t\Big\}\Big)^{1/q}\lesssim
\frac{1}{t}\Big(\int_{\mathbb{R}^d}|f(x)|^p\mu(x)^{p/q}\nu(x)dx\Big)^{1/p}.
\end{align*}
\end{theorem}
\begin{remark}
When $\nu=p=1$, Theorems \ref{mix schrodinger fractional maximal} and \ref{mix schrodinger fractional} are just the conclusions in \cite{BonHS,Tang}. Hence, Theorems \ref{mix schrodinger fractional maximal} and \ref{mix schrodinger fractional} cover the results in \cite{BonHS,Tang}.
\end{remark}

We organize the rest of the paper as follows. In Section 2, we will give some preliminaries. In Section 3, we will prove Theorem $\ref{schrodinger fractional weak type}$ and the proofs of Theorems \ref{mix schrodinger fractional maximal}, \ref{mix schrodinger fractional} will be given in Section 4.

Throughout the rest of the paper, we denote $f\lesssim g$, $f\thicksim g$ if $f\leq Cg$ and $f\lesssim g \lesssim f$, respectively. For any ball $B:=B(x_B,r_B)\subset \mathbb{R}^d$ and $\sigma>0$, $\chi_B$ represents the characteristic function of $B$ and $\sigma B$ means $B(x_B,\sigma r_B)$.

\section{Preliminaries}
In this section, we introduce some basic definitions and necessary lemmas. We first recall the definition of critical radius function. A function $\rho:\mathbb{R}^d\rightarrow(0,\infty)$ is called critical radius function if there exist constants $C_0$ and $N_0$ such that for any $x,y\in\mathbb{R}^d$,
\begin{align}\label{critical radius function}
C_0^{-1}\rho(x)\Big(1+\frac{|x-y|}{\rho(x)}\Big)^{-N_0}\leq\rho(y)\leq C_0\rho(x)
\Big(1+\frac{|x-y|}{\rho(x)}\Big)^{\frac{N_0}{N_0+1}}.
\end{align}
In particular, let $d\geq3$, $V\in RH_\sigma$ ($\sigma>d/2$) be a non-negative function, not identically zero. Shen \cite{Sh} proved that the function
\begin{align*}
\rho(x)=\sup\Big\{r>0:\frac{1}{r^{d-2}}\int_{B(x,r)}V(x)dx\leq1\Big\}
\end{align*}
is a critical radius function. The following covering lemma is very useful in our proof.
\begin{lemma}{\rm(cf. \cite{DzZ})}\label{covering lemma}
There exists a sequence of points $x_j$ in $\mathbb{R}^d$, so that the family $\{Q_j:=Q(x_j,\rho(x_j))\}_{j\in\mathbb{Z}^+}$ satisfies:\\
\medskip
{\rm (i)} $\bigcup_{j\in\mathbb{Z}^+}Q_j=\mathbb{R}^d$;\\
\medskip
{\rm (ii)} For every $\sigma\geq1$, there exists constants $C,N>0$ such that for any $x\in\mathbb{R}^d$, $\sum_{j\in\mathbb{Z}^+}\chi_{\sigma Q_j}(x)\leq C\sigma^N$.
\end{lemma}

Next, we introduce new classes of weights, which are extension of $A_{p,q}$ weights associated to critical radius function introduced in \cite{Tang}.
\begin{definition}
A weight $\omega$ is a locally integrable function on $\mathbb{R}^d$ that verifies $0<\omega(x)<\infty$ almost everywhere. Let $\rho$ be a critical radius function, $\theta\geq0$. We say that $\omega\in A_{p,q}^{\rho,\theta}(\mu)$ $(1<p,q<\infty)$ if
\begin{align*}
&[\omega]_{A_{p,q}^{\rho,\theta}(\mu)}:\\
&\quad=\sup_{Q}\Big(\frac{1}{\psi_\theta(Q)\mu(Q)}\int_{Q}\omega(x)^q\mu(x)
dx\Big)\Big(\frac{1}{\psi_\theta(Q)\mu(Q)}\int_{Q}\omega(x)^{-p'}\mu(x)dx\Big)^{q/p'}<\infty.
\end{align*}
We say that $w\in A_{1,q}^{^{\rho,\theta}}(\mu)$ if
$$[\omega]_{A_{1,q}^{^{\rho,\gamma}}(\mu)}:=\sup_B\Big(\frac{1}{\psi_\theta(Q)\mu(Q)}\int_{Q}
\omega(x)^q\mu(x)dx\Big)\|\omega^{-1}\|_{L^\infty(Q)}^{q}<\infty,$$
where the supremum is taken over all cubes $Q:=Q(x_Q,r_Q)\subset\mathbb{R}^d$.
\end{definition}
Obviously, $A_{p,q}^{\rho,\theta}(1)$ coincides $A_{p,q}^{\rho,\theta}(\mathbb{R}^d)$ (see its definition in \cite{Tang}) when $\mu=1$. If $1/q=1/p-\alpha/n$, one can also check that $\omega\in A_{p,q}^{\rho,\theta}(\mu)$ if and only if $\omega^q\in A_{1+q/p'}^{\rho,\theta}(\mu)$ ($p>1$) and $\omega\in A_{1,q}^{\rho,\theta}(\mu)$ if and only if $\omega^q\in A_{1}^{\rho,\theta}(\mu)$, where $A_{p}^{\rho,\theta}(\mu)$ and $A_{1}^{\rho,\theta}(\mu)$ are the collections of weights that satisfy
\begin{align*}
&[\omega]_{A_p^{\rho,\theta}(\mu)}:\\
&\quad=\sup_{Q}\Big(\frac{1}{\psi_\theta(Q)\mu(Q)}\int_{Q}\omega(x)\mu(x)dx\Big)
\Big(\frac{1}{\psi_\theta(Q)\mu(Q)}\int_{Q}\omega(x)^{1-p'}\mu(x)dx\Big)^{p-1}<\infty,
\end{align*}
and
\begin{align*}
[\omega]_{A_p^{\rho,\theta}(\mu)}:=\sup_{Q}\Big(\frac{1}{\psi_\theta(Q)\mu(Q)}\int_{Q}\omega(x)\mu(x)dx\Big)
\|\omega^{-1}\|_{L^\infty(Q)}<\infty,
\end{align*}
respectively.
For $1\leq p<\infty$ and $1<q<\infty$, we define
$$A_p^\rho(\mu)=\cup_{\theta\geq0}A_{p}^{\rho,\theta}(\mu),~
A_{p,q}^\rho(\mu)=\cup_{\theta\geq0}A_{p,q}^{\rho,\theta}(\mu),~
A_\infty^\rho(\mu)=\cup_{p\geq1}A_{p}^{\rho}(\mu).$$

In our proofs, we also need the following $\rho$-localized weights.
\begin{definition}\label{local weight}
Let $\rho$ be a critical radius function. We say a weight $\omega\in A_{p}^{\rho,loc}(\mathbb{R}^d)$ if
\begin{equation*}
[\omega]_{A_{p}^{\rho,loc}(\mathbb{R}^d)}:=\sup_{Q\in\mathcal{Q}_\rho}\Big(\frac{1}{|Q|}\int_Q
\omega(x)dx\Big)\Big(\frac{1}{|Q|}\int_Q\omega(x)^{-1/(p-1)}dx\Big)^{p-1}<\infty,~1<p<\infty,
\end{equation*}
and we say that $\omega\in A_{1}^{\rho,loc}(\mathbb{R}^n)$ if
\begin{equation*}
[\omega]_{A_{1}^{\rho,loc}(\mathbb{R}^d)}:=\sup_{Q\in\mathcal{Q}_\rho}\Big(\frac{1}{|Q|}\int_Q
\omega(x)dx\Big)\|\omega^{-1}\|_{L^\infty(Q)}<\infty,
\end{equation*}
where $\mathcal{Q}_\rho:=\{Q(x,r):r\leq\rho(x)\}$. Similarly, $A_{p,q}^{\rho,loc}(\mathbb{R}^d)$ and $A_{1,q}^{\rho,loc}(\mathbb{R}^d)$ is defined by
\begin{align*}
[\omega]_{A_{p,q}^{\rho,loc}(\mathbb{R}^d)}:=\sup_{Q\in\mathcal{Q}_\rho}\Big(\frac{1}{|Q|}\int_{Q}\omega(x)^q
dx\Big)\Big(\frac{1}{|Q|}\int_{Q}\omega(x)^{-p'}dx\Big)^{q/p'}<\infty.
\end{align*}
and
\begin{align*}
[\omega]_{A_{1,q}^{\rho,loc}(\mathbb{R}^d)}:=\sup_{Q\in\mathcal{Q}_\rho}\Big(\frac{1}{|Q|}\int_{Q}
\omega(x)^qdx\Big)\|\omega^{-1}\|_{L^\infty(Q)}^{q}<\infty,
\end{align*}
respectively.
\end{definition}
Let $Q_0$ be a cube. If we replace $Q\in\mathcal{Q}_\rho$ by $Q\subset Q_0$ in Definition \ref{local weight}, then we say a weight defined on $Q_0$ belongs to $A_p(Q_0)$ ($p\geq1$).
 
We give a remark about these classes of weights.
\begin{remark}
{\rm (1)} It is trivial that $A_{p}^{\rho,\gamma}(\mathbb{R}^d)\subset A_{p}^{\rho,loc}(\mathbb{R}^d)$ and $[\omega]_{A_{p}^{\rho,loc}(\mathbb{R}^d)}\lesssim[\omega]_{A_{p}^{\rho,\gamma}(\mathbb{R}^d)}$; \\
{\rm (2)} In \cite{BonHS}, the authors proved that $$A_p^{\beta\rho,loc}(\mathbb{R}^d)=A_p^{\rho,loc}(\mathbb{R}^d)$$
with $[\omega]_{A_p^{\rho,loc}(\mathbb{R}^d))}\sim[\omega]_{A_p^{\beta\rho,loc}(\mathbb{R}^d)}$ for any $\beta>1$.\\
{\rm (3)} Analogy with the definition of $A_{p,q}^{\rho,\theta}(\mu)$, we can also define $A_{p}^{\rho,loc}(\mu)$ and $A_p(Q_0,\mu)$.
\end{remark}
Finally, we recall the following lemma, which is concerned with the extension of weights.
\begin{lemma}{\rm(cf. \cite{BonHS})}\label{extension}
Given a cube $Q_0$ in $\mathbb{R}^d$ and a weight $\omega_0\in A_p(Q_0)$, $1\leq p<\infty$, then $\omega_0$ has an extension $\omega\in A_p(\mathbb{R}^d)$ such that for any $x\in Q_0$, $\omega_0(x)=\omega(x)$ and $[\omega]_{A_p(Q_0)}\sim[\omega]_{A_p(\mathbb{R}^d)}$, where the implicit constants are independent of $\omega_0$ and $p$.
\end{lemma}

\section{Quantitative weighted estimates for fractional integral operators associated to Schr\"{o}dinger operator}
In this section, we give the proof of Theorem \ref{schrodinger fractional weak type}. Before this, we first establish the following quantitative weighted estimates for $\rho$-localized classical fractional integral operators.
\begin{lemma}\label{local chuli}
Let $d\geq3$, $\rho$ be a critical radius function and $B_x:=B(x,\rho(x))$ with $x\in\mathbb{R}^d$. Suppose that $0<\alpha<d$, $1\leq p<d/\alpha$ and $1/q=1/p-\alpha/d$.\\
\medskip
{\rm (i)} If $\omega\in A_{p,q}^{\rho,loc}(\mathbb{R}^d)$, there holds
\begin{align*}
\|I_\alpha(f\chi_{B_x})\|_{L^{q,\infty}(\omega^q)}\lesssim
[\omega]_{A_{p,q}^{\rho,loc}(\mathbb{R}^d)}^{1-\frac{\alpha}{d}}\|f\|_{L^p(\omega^p)}.
\end{align*}
\medskip
{\rm (ii)} If $\omega\in A_{p,q}^{\rho,loc}(\mathbb{R}^d)$, there holds
\begin{align*}
\|I_\alpha(f\chi_{B_x})\|_{L^q(\omega^q)}\lesssim[\omega]_{A_{p,q}^{\rho,loc}(\mathbb{R}^d)}
^{(1-\frac{\alpha}{d})\max\{1,\frac{p'}{q}\}}\|f\|_{L^p(\omega^p)}.
\end{align*}
\end{lemma}
\begin{proof}
(i). We only prove the case $p=1$ since $p>1$ is similar. Let $\tau=1+C_02^{\frac{N_0}{N_0+1}}$, where $C_0$ and $N_0$ are given in \eqref{critical radius function}. Let $\{B_j:B(x_j,\rho(x_j))\}_{j\in\mathbb{N}}$ be the family of balls given by Lemma \ref{covering lemma}. Denote $\widetilde{B_j}=\tau B_j$, then $B_x\subset\widetilde{B_j}$ for any $j\in\mathbb{N}$ and $x\in B_j$. To see this, for any $y\in B_x$, in virtue of \eqref{critical radius function}, we have
\begin{align*}
|y-x_j|&\leq|y-x|+|x-x_j|\leq\rho(x)+\rho(x_j)\\
&\leq C_0\rho(x_j)\Big(1+\frac{\rho(x_j)}{\rho(x_j)}\Big)^{\frac{N_0}{N_0+1}}+\rho(x_j)=\tau\rho(x_j).
\end{align*}
Now, we claim that for any $\omega\in A_{1,q}^{\rho,loc}(\mathbb{R}^d)$ and $j\in\mathbb{N}$, there holds $\omega^q|_{\widetilde{B_j}}\in A_1(\widetilde{B_j})$ and
\begin{align}\label{local claim}
[\omega^q|_{\widetilde{B_j}}]_{A_1(\widetilde{B_j})}\lesssim[\omega]_{A_{1,q}^{\rho,loc}(\mathbb{R}^d)}.
\end{align}
In fact, for any ball $B:=B(x_B,r_B)\subset\widetilde{B_j}$ , we demonstrate it by considering two cases:\\
\textbf{Case 1. $r_B\leq\tau\rho(x_B)$:}\quad one can check that
\begin{align*}
[\omega]_{A_{1,q}^{\rho,loc}(\mathbb{R}^d)}=[\omega^q]_{A_1^{\rho,loc}(\mathbb{R}^d)}\sim
[\omega^q]_{A_1^{\tau\rho,loc}(\mathbb{R}^d)},
\end{align*}
which further implies that
\begin{align*}
\Big(\frac{1}{|B|}\int_B\omega(x)^qdx\Big)\Big(\inf_{B\ni x}\omega(x)^q\Big)^{-1}\leq[\omega^q]_{A_1^{\tau\rho,loc}(\mathbb{R}^d)}\sim
[\omega]_{A_{1,q}^{\rho,loc}(\mathbb{R}^d)}.
\end{align*}
This shows \eqref{local claim}.\\
\textbf{Case 2. $r_B>\tau\rho(x_B)$:} In this case, it is easy to see that $B(x_B,\tau\rho(x_B))\subset B\subset \widetilde{B_j}$. Again by \eqref{critical radius function} and $|x_j-x_B|\leq\tau\rho(x_j)$, we have $\rho(x_B)\sim\rho(x_j)$. Combing these facts, we deduce that $|B|\sim|\widetilde{B_j}|$ and
\begin{align*}
\Big(\frac{1}{|B|}\int_B\omega(x)^qdx\Big)(\inf_{B\ni x}\omega(x)^q)^{-1}&\lesssim
\Big(\frac{1}{|\widetilde{B_j}|}\int_{\widetilde{B_j}}\omega(x)^qdx\Big)\Big(\inf_{\widetilde{B_j}\ni x}\omega(x)^q\Big)^{-1}\\
&\leq[\omega^q]_{A_1^{\tau\rho,loc}(\mathbb{R}^d)}\sim[\omega^q]_{A_{1,q}^{\rho,loc}(\mathbb{R}^d)}.
\end{align*}
This also verifies \eqref{local claim}.

Now we return to the proof of our lemma. For any $j\in\mathbb{N}$, in virtue of Lemma \ref{extension} and \eqref{local claim}, $\omega^q|_{\widetilde{B_j}}$ admits an extension $\omega_j$ on $\mathbb{R}^d$, which satisfies $\omega_j\in A_1(\mathbb{R}^d)$ and
\begin{align*}
[\omega_j]_{A_1(\mathbb{R}^d)}\sim[\omega^q|_{\widetilde{B_j}}]_{A_1(\widetilde{B_j})}
\lesssim[\omega]_{A_{1,q}^{\rho,loc}(\mathbb{R}^d)},
\end{align*}
where the implicit constant is independent of $j$. By making use of Theorem \ref{classical fractional}, we have that for any $\omega\in A_{1,q}(\mathbb{R}^d)$,
\begin{align*}
\|I_\alpha f\|_{L^{q,\infty}(\omega^{q})}\lesssim[\omega]_{A_{1,q}(\mathbb{R}^d)}^{1-\frac{\alpha}{d}}
\|f\|_{L^1(\omega)}.
\end{align*}
Then by Lemma \ref{covering lemma}, we have
\begin{align*}
&t\Big(\omega^q(\{x\in\mathbb{R}^d:I_\alpha(f\chi_{B_x})(x)>t\})\Big)^{1/q}\\
&\quad=t\Big(\omega^q(\{x\in\bigcup_jB_j:I_\alpha(f\chi_{B_x})(x)>t\})\Big)^{1/q}\\
&\quad\leq\sum_jt\Big(\omega^q(\{x\in B_j:I_\alpha(f\chi_{B_x})(x)>t\})\Big)^{1/q}\\
&\quad\leq\sum_jt\Big(\omega_j(\{x\in B_j:I_\alpha(f\chi_{B_x})(x)>t\})\Big)^{1/q}\\
&\quad\lesssim\sum_j[\omega_j]_{A_1(\mathbb{R}^d)}^{1-\frac{\alpha}{d}}
\int_{\widetilde{B_j}}|f(x)|\omega_j(x)^{1/q}dx\\
&\quad\lesssim[\omega]_{A_{1,q}^{\rho,loc}(\mathbb{R}^d)}^{1-\frac{\alpha}{d}}
\sum_j\int_{\widetilde{B_j}}|f(x)|\omega(x)dx\\
&\quad\lesssim[\omega]_{A_{1,q}^{\rho,loc}(\mathbb{R}^d)}^{1-\frac{\alpha}{d}}
\int_{\mathbb{R}^d}|f(x)|\omega(x)dx.
\end{align*}
This completes the proof of (i) of Lemma \ref{local chuli}.

(ii). We still use the notations given in (i). Similar to \eqref{local claim}, we get
\begin{align*}
[\omega^q|_{\widetilde{B_j}}]_{A_{1+q/p'}(\widetilde{B_j})}\lesssim
[\omega]_{A_{p,q}^{\rho,loc}(\mathbb{R}^d)}.
\end{align*}
Furthermore, according to Lemma \ref{extension}, $\omega^q|_{\widetilde{B_j}}$ admits an extension $\omega_j\in A_{1+q/p'}(\mathbb{R}^d)$, which satisfies $\omega_j(x)=\omega^q|_{\widetilde{B_j}}$ for each $x\in\widetilde{B_j}$ and
\begin{align*}
[\omega_j]_{A_{1+q/p'}(\mathbb{R}^d)}\sim[\omega^q|_{\widetilde{B_j}}]_{A_{1+q/p'}(\widetilde{B_j})}
\lesssim[\omega]_{A_{p,q}^{\rho,loc}(\mathbb{R}^d)}.
\end{align*}
Therefore, in virtue of Theorem \ref{classical fractional}, Lemma \ref{covering lemma} and $p<q$, we deduce that
\begin{align*}
\|I_\alpha(f\chi_{B_x})\|_{L^q(\omega^q)}^{p}&\leq\Big(\sum_j\int_{\widetilde{B_j}}
I_\alpha(f\chi_{B_x})(x)^q\omega(x)^qdx\Big)^{p/q}\\
&=\Big(\sum_j\int_{\widetilde{B_j}}I_\alpha(f\chi_{B_x})(x)^q\omega_j(x)dx\Big)^{p/q}\\
&\leq\Big(\sum_j\int_{\mathbb{R}^d}I_\alpha(f\chi_{B_x})(x)^q\omega_j(x)dx\Big)^{p/q}\\
&\lesssim\Big(\sum_j[\omega_j]_{A_{1+q/p'}}^{q\big(1-\frac{\alpha}{d}\big)\max\{1,\frac{p'}{q}\}}
\Big(\int_{\mathbb{R}^d}|f(x)|^p\chi_{B_x}(x)\omega_j(x)^{p/q}dx\Big)^{q/p}\Big)^{p/q}\\
&\lesssim[\omega]_{A_{p,q}^{\rho,loc}(\mathbb{R}^d)}^{p\big(1-\frac{\alpha}{d}\big)\max\{1,\frac{p'}{q}\}}
\Big(\sum_j\Big(\int_{\widetilde{B_j}}|f(x)|^p\omega(x)^pdx\Big)^{q/p}\Big)^{p/q}\\
&\lesssim[\omega]_{A_{p,q}^{\rho,loc}(\mathbb{R}^d)}^{p\big(1-\frac{\alpha}{d}\big)\max\{1,\frac{p'}{q}\}}
\int_{\mathbb{R}^d}|f(x)|^p\omega(x)^pdx.
\end{align*}
\end{proof}

To achieve our main results, we also need the following lemma.
\begin{lemma}{\rm(cf. \cite{WW})}\label{weak type schrodinger maximal operator}
Let $d\geq3$, $\theta\geq0$ and $\rho$ be a critical function. Assume that $0\leq\alpha<d$, $1\leq p<d/\alpha$ and $1/q=1/p-\alpha/d$. Then for $\omega\in A_{p,q}^{\rho,\theta}(\mathbb{R}^d)$ and $f\in L^p(\omega^p)$,
\begin{equation*}
\|M_{\alpha}^{\rho,\theta}f\|_{L^{q,\infty}(\omega^q)}\lesssim [\omega]_{A_{p,q}^{\rho,\theta}(\mathbb{R}^d)}^{1/q}\|f\|_{L^p(\omega^p)}.
\end{equation*}
\end{lemma}

Now, we are in the position to prove Theorem \ref{schrodinger fractional weak type}.
\begin{proof}[Proof of Theorem \ref{schrodinger fractional weak type}]
(i). For each $f\in C_c^\infty(\mathbb{R}^d)$ and $x\in \mathbb{R}^d$. We split $f$ as $f=f_1+f_2$, where $f_1:=f\chi_{B_x}$, $B_x:=B(x,\rho(x))$. It follows that
\begin{align*}
\mathcal{L}^{-\alpha/2}f(x)=\mathcal{L}^{-\alpha/2}f_1(x)+\mathcal{L}^{-\alpha/2}f_2(x).
\end{align*}
Denote the kernel of $e^{-t\mathcal{L}}$ by $p_t(x,y)$. It is well known that $p_t(x,y)\leq h_t(x-y)$, where $h_t(x-y)$ is the classical heat kernel. Then we have the pointwise inequality
\begin{align}\label{local dominate}
|\mathcal{L}^{-\alpha/2}f_1(x)|\leq I_\alpha(|f_1|)(x).
\end{align}
This together with Lemma \ref{local chuli}, allows us to get that
\begin{align}\label{local estimate}
\|\mathcal{L}^{-\alpha/2}f_1\|_{L^q(\omega^q)}&\leq\|I_\alpha(|f_1|)\|_{L^q(\omega^q)}\\
&\lesssim[\omega]_{A_{p,q}^{\rho,loc}(\mathbb{R}^d)}
^{(1-\frac{\alpha}{d})\max\{1,\frac{p'}{q}\}}\|f\|_{L^p(\omega^p)}\nonumber\\
&\leq[\omega]_{A_{p,q}^{\rho,\gamma/3}(\mathbb{R}^d)}
^{(1-\frac{\alpha}{d})\max\{1,\frac{p'}{q}\}}\|f\|_{L^p(\omega^p)}\nonumber.
\end{align}

Next, we consider $\mathcal{L}^{-\alpha/2}f_2$. Given $N>0$, recall that there is a constant $C_N$ such that for any $x,y\in\mathbb{R}^d$,
\begin{align*}
p_t(x,y)\leq C_Nt^{-d/2}e^{-\frac{|x-y|^2}{5t}}\Big(1+\frac{\sqrt{t}}{\rho(x)}+\frac{\sqrt{t}}{\rho(y)}\Big)^{-N},
\end{align*}
see \cite{Ku}. Thus,
\begin{align*}
|\mathcal{L}^{-\alpha/2}f_2(x)|&\leq\int_0^\infty\int_{B_x^c}p_t(x,y)|f(y)|dyt^{\alpha/2-1}dt\\
&\lesssim\int_0^\infty\int_{B_x^c}e^{-\frac{|x-y|^2}{5t}}
\Big(1+\frac{\sqrt{t}}{\rho(x)}\Big)^{-N}|f(y)|dyt^{\alpha/2-d/2-1}dt\\
&\lesssim\int_{B_x^c}\frac{|f(y)|}{|x-y|^M}dy\int_0^\infty\Big(1+\frac{\sqrt{t}}{\rho(x)}\Big)^{-N}
t^{M/2+\alpha/2-d/2-1}dt,
\end{align*}
where we use $e^{-s}\lesssim s^{-M/2}$ for any $M>0$ in the last inequality. Let $\theta\geq0$, note that for $N\geq M>d-\alpha+\theta$,
\begin{align*}
&\int_0^\infty\Big(1+\frac{\sqrt{t}}{\rho(x)}\Big)^{-N}
t^{M/2+\alpha/2-d/2-1}dt\\
&\quad\leq\int_{0}^{\rho(x)^2}t^{M/2+\alpha/2-d/2-1}dt+\rho(x)^N\int_{\rho(x)^2}^{\infty}
t^{-N/2+M/2-d/2+\alpha/2-1}dt\\
&\quad\lesssim\rho(x)^{M-d+\alpha}.
\end{align*}
Then we can continue the estimate with
\begin{align}\label{global dominate}
|\mathcal{L}^{-\alpha/2}f_2(x)|&\lesssim\sum_{k=0}^{\infty}\int_{2^{k+1}B_x\backslash 2^kB_x}\rho(x)^{M-d+\alpha}\frac{|f(y)|}{(2^k\rho(x))^M}dy\\
&\sim\sum_{k=0}^{\infty}2^{-kM}\frac{2^{(k+1)(d-\alpha+\theta)}}
{(2^{k+1}\rho(x))^{d-\alpha}\Big(1+\frac{2^{k+1}\rho(x)}{\rho(x)}\Big)^{\theta}}\int_{2^{k+1}B_x}|f(y)|dy
\nonumber\\
&\lesssim M_\alpha^{\rho,\theta} f(x)\sum_{k=0}^\infty2^{-k(M-d+\alpha-\theta)}\sim M_\alpha^{\rho,\theta} f(x).\nonumber
\end{align}
In view of Theorem \ref{schrodinger fractional strong type}, we obtain
\begin{align*}
\|\mathcal{L}^{-\alpha/2}f_2\|_{L^q(\omega^q)}\lesssim\|M_\alpha^{\rho,\theta} f\|_{L^q(\omega^q)}
\lesssim[\omega]_{A_{p,q}^{\rho,\gamma/3}(\mathbb{R}^d)}^{(1-\frac{\alpha}{n})}
\|f\|_{L^p(\omega^p)},
\end{align*}
This combines with \eqref{local estimate}, we get
\begin{align*}
\|\mathcal{L}^{-\alpha/2}f\|_{L^q(\omega^q)}\lesssim[\omega]_{A_{p,q}^{\rho,\gamma/3}(\mathbb{R}^d)}
^{(1-\frac{\alpha}{n})\max\{1,\frac{p'}{q}\}}\|f\|_{L^p(\omega^p)}, ~1<p<n/\alpha.
\end{align*}

(ii). By \eqref{local dominate} and \eqref{global dominate}, we have
\begin{align}\label{fractional zong}
\|\mathcal{L}^{-\alpha/2}f\|_{L^{q,\infty}(\omega^q)}&\leq
\|\mathcal{L}^{-\alpha/2}f_1\|_{L^{q,\infty}(\omega^q)}+
\|\mathcal{L}^{-\alpha/2}f_2\|_{L^{q,\infty}(\omega^q)}\\
&\leq\|I_\alpha(|f_1|)\|_{L^{q,\infty}(\omega^q)}+\|M_\alpha^{\rho,\theta}f\|_{L^{q,\infty}
(\omega^q)}\nonumber
\end{align}
Since $[\omega]_{A_{p}^{\rho,loc}(\mathbb{R}^d)}\lesssim[\omega]_{A_{p}^{\rho,\theta}(\mathbb{R}^d)}$, Lemma \ref{local chuli} yields that
\begin{align}\label{fractional fen1}
\|I_\alpha(|f_1|)\|_{L^{q,\infty}(\omega^q)}\lesssim
[\omega]_{A_{p,q}^{\rho,\theta}(\mathbb{R}^d)}^{1-\frac{\alpha}{d}}\|f\|_{L^p(\omega^p)}.
\end{align}
On the other hand, Lemma \ref{weak type schrodinger maximal operator} shows that
\begin{align}\label{fractional fen2}
\|M_\alpha^{\rho,\theta}f\|_{L^{q,\infty}
(\omega^q)}\lesssim
[\omega]_{A_{p,q}^{\rho,\theta}(\mathbb{R}^d)}^{1-\frac{\alpha}{d}}\|f\|_{L^p(\omega^p)}.
\end{align}
Hence, by \eqref{fractional zong}-\eqref{fractional fen2}, we get the expected result.
\end{proof}

\section{Weighted mixed weak type inequalities for fractional type integral operators associated to Schr\"{o}dinger operator}
In this section, we prove Theorems \ref{mix schrodinger fractional maximal} and \ref{mix schrodinger fractional}. To prove Theorem \ref{mix schrodinger fractional maximal}, we adapt some ideals in \cite{BeCaPr}. We need the following lemmas.
\begin{lemma}{\rm(cf. \cite{BePrQu})}\label{schrodinger maximal mix weak type}
Let $\mu_1\in A_1^\rho(\mathbb{R}^d)$ and $\nu_1\in A_\infty^\rho(\mu_1)$. Then there exists $\theta\geq0$ such that for every positive $t$,
\begin{align*}
\mu_1\nu_1\Big(\Big\{x\in\mathbb{R}^d:\frac{M^{\rho,\theta}(f_1\nu_1)(x)}{\nu_1(x)}>t\Big\}\Big)
\leq\frac{C}{t}\int_{\mathbb{R}^d}f_1(x)\mu_1(x)\nu_1(x)dx.
\end{align*}
\end{lemma}
\begin{lemma}\label{maximal pointwise domination}
Let $0<\alpha<d$, $1\leq p<d/\alpha$, $1/q=1/p-\alpha/d$, $s=1+q/p'$, $\rho$ be a critical radius function and $\theta\geq0$. For each non-negative function $\omega_0$ and $f_0\in L^{p}(\mathbb{R}^d)$, there holds
\begin{equation*}
M_{\alpha}^{\rho,\theta}(f_0\omega_0^{-1})(x)\leq[M^{\rho,\theta}(f_0^{p/s}\omega_0^{-q/s})(x)]
^{s/q}\Big(\int_{\mathbb{R}^d}f_0(x)^{p}dx\Big)^{\alpha/d},
\end{equation*}
and
\begin{equation*}
M_{\alpha}^{\rho,\theta}(f_0\omega_0^{-1})(x)\leq[M^{\rho,\theta}(f_0^{p}\omega_0^{-q})(x)]
^{1/q}\Big(\int_{\mathbb{R}^d}f_0(x)^{p}dx\Big)^{\alpha/d}.
\end{equation*}
\end{lemma}
\begin{proof}
Fix $x\in\mathbb{R}^d$. Let $Q:=Q(x_Q,r_Q)$ be a cube and $x\in Q$. Applying H\"{o}lder's inequality with exponents $d/(d-\alpha)$ and $d/\alpha$, we get
\begin{align*}
&\frac{1}{\psi_\theta(Q)|Q|^{1-\frac{\alpha}{d}}}\int_Qf_0\omega_0^{-1}
\nonumber\\
&\quad=\frac{1}{\psi_\theta(Q)|Q|^{1-\frac{\alpha}{d}}}\int_Q
(f_0^{p/s}\omega_0^{-q/s})^{1-\frac{\alpha}{d}}
(f_0^{p/s}\omega_0^{-q/s})^{\frac{s}{p}+\frac{\alpha}{d}-1}\omega_0
^{q\alpha/d}\\
&\quad\leq\frac{1}{\psi_\theta(Q)|Q|^{1-\frac{\alpha}{d}}}
\Big(\int_Qf_0^{p/s}\omega_0^{-q/s}\Big)^{1-\frac{\alpha}{d}}
\Big(\int_Q(f_0^{p/s}\omega_0^{-q/s})^{(\frac{s}{p}+\frac{\alpha}{d}-1)\frac{d}{\alpha}}
\omega_0^q\Big)^{\alpha/d}.
\end{align*}
Note that
$$(f_0^{p/s}\omega_0^{-q/s})^{(\frac{s}{p}+\frac{\alpha}{d}-1)\frac{d}{\alpha}}
\omega_0^{q}=f_0^{p}, ~(1-\alpha/d)/s=1/q.$$
We conclude that
\begin{align*}
&\frac{1}{\psi_\theta(Q)|Q|^{1-\frac{\alpha}{d}}}\int_Qf_0\omega_0^{-1}\\
&\quad\leq\Big(\frac{1}{\psi_\theta(Q)|Q|}
\int_Qf_0^{p/s}\omega_0^{-q/s}\Big)^{s/q}
\Big(\frac{1}{\psi_\theta(Q)}\int_Qf_0^{p}\Big)^{\alpha/d}\\
&\quad\leq\Big(\frac{1}{\psi_\theta(Q)|Q|}
\int_Qf_0^{p/s}\omega_0^{-q/s}\Big)^{s/q}
\Big(\int_Qf_0^{p}\Big)^{\alpha/d}\\
&\quad\leq[M^{\rho,\theta}(f_0^{p/s}\omega_0^{-q/s})(x)]^{s/q}
\Big(\int_{\mathbb{R}^d}f_0^{p}\Big)^{\alpha/d}.
\end{align*}
Taking a supremum over all cubes contained $x$, we arrive at the first conclusion.

To achieve the second result. If $s=1$, the second result follows by the first conclusion. If $s>1$, by using
$$\frac{1}{\psi_\theta(Q)}\leq\frac{1}{\psi_\theta(Q)^{1/p}}=\frac{1}{\psi_\theta(Q)^{1/q+\alpha/d}},$$
and
\begin{align*}
\Big(\frac{1}{|Q|}\int_Qf_0^{p/s}\omega_0^{-q/s}\Big)^{1-\frac{\alpha}{d}}=
\Big(\frac{1}{|Q|}\int_Qf_0^{p/s}\omega_0^{-q/s}\Big)^{s/q}\leq
\Big(\frac{1}{|Q|}\int_Qf_0^{p}\omega_0^{-q}\Big)^{1/q}.
\end{align*}
We deduce that
\begin{align*}
&\frac{1}{\psi_\theta(Q)|Q|^{1-\frac{\alpha}{d}}}\int_Qf_0\omega_0^{-1}\\
&\quad\leq\frac{1}{\psi_\theta(Q)}
\Big(\frac{1}{|Q|}\int_Qf_0^{p/s}\omega_0^{-q/s}\Big)^{1-\frac{\alpha}{d}}
\Big(\int_Qf_0^{p}\Big)^{\alpha/d}\\
&\quad\leq\Big(\frac{1}{\psi_\theta(Q)|Q|}
\int_Qf_0^{p}\omega_0^{-q}\Big)^{1/q}
\Big(\frac{1}{\psi_\theta(Q)}\int_Qf_0^{p}\Big)^{\alpha/d}\\
&\quad\leq[M^{\rho,\theta}(f_0\omega_0^{-q})(x)]^{1/q}\Big(\int_{\mathbb{R}^d}f_0^{p}\Big)^{\alpha/d}.
\end{align*}
Taking a supremum over all cubes contained $x$, we get the second conclusion.
\end{proof}

\begin{proof}[Proof of Theorem \ref{mix schrodinger fractional maximal}]
Without loss of generality, we shall assume that $f\geq0$. We prove it by considering two cases:\\
\textbf{Case 1. $p=1$:}\quad It is enough to consider $f\mu^{1/q}\nu\in L^1(\mathbb{R}^d)$, since in the other case, there is nothing to prove. Applying Lemma \ref{maximal pointwise domination} with $f_0=f\mu^{1/q}\nu$, $\omega_0=\mu^{1/q}$, we have
\begin{align*}
&\mu\nu^q\Big(\Big\{x\in\mathbb{R}^d:\frac{M_\alpha^{\rho,\theta}(f\nu)(x)}{\nu(x)}>t\Big\}\Big)^{1/q}\\
&\quad=\mu\nu^q\Big(\Big\{x\in\mathbb{R}^d:\frac{M_\alpha^{\rho,\theta}(f\mu^{1/q}\nu\mu^{-1/q})(x)}
{\nu(x)}>t\Big\}\Big)^{1/q}\\
&\quad\leq\mu\nu^q\Big(\Big\{x\in\mathbb{R}^d:\frac{M^{\rho,\theta}(f\mu^{1/q}\nu\mu^{-1})(x)^{1/q}
\|f\mu^{1/q}\nu\|_{L^1(\mathbb{R}^d)}^{\alpha/d}}
{\nu(x)}>t\Big\}\Big)^{1/q}\\
&\quad=\mu\nu^q\Big(\Big\{x\in\mathbb{R}^d:\frac{M^{\rho,\theta}(f\mu^{1/q}\nu\mu^{-1})(x)}
{\nu(x)^q}>\frac{t^q}{\|f\mu^{1/q}\nu\|_{L^1(\mathbb{R}^d)}^{q\alpha/d}}\Big\}\Big)^{1/q}.
\end{align*}
When $p=1$, $\mu\nu^{-q/p'}\in A_1^\rho(\mathbb{R}^d)$ and $\nu^q\in A_{\infty}^{\rho}(\mu\nu^{-q/p'})$ reduce to $\mu\in A_1^\rho(\mathbb{R}^d)$ and $\nu^q\in A_{\infty}^{\rho}(\mu)$. Then by Lemma \ref{schrodinger maximal mix weak type} with $f_1=f\mu^{1/q}\nu\mu^{-1}\nu^{-q}$, $\nu_1=\nu^q$ and $\mu_1=\mu$,
\begin{align*}
&\mu\nu^q\Big(\Big\{x\in\mathbb{R}^d:\frac{M_\alpha^{\rho,\theta}(f\nu)(x)}{\nu(x)}>t\Big\}\Big)^{1/q}\\
&\quad\leq\mu\nu^q\Big(\Big\{x\in\mathbb{R}^d:\frac{M^{\rho,\theta}(f\mu^{1/q}\nu\mu^{-1}\nu^{-q}\nu^q)(x)}
{\nu(x)^q}>\frac{t^q}{\|f\mu^{1/q}\nu\|_{L^1(\mathbb{R}^d)}^{q\alpha/d}}\Big\}\Big)^{1/q}\\
&\quad\lesssim\frac{\|f\mu^{1/q}\nu\|_{L^1(\mathbb{R}^d)}^{\alpha/d}}{t}
\Big(\int_{\mathbb{R}^d}f(x)\mu(x)^{1/q}\nu(x)dx\Big)^{1/q}\\
&\quad=\frac{1}{t}\int_{\mathbb{R}^d}f(x)\mu(x)^{1/q}\nu(x)dx,
\end{align*}
where we have used $\alpha/d+1/q=1$ in the last equality.

\textbf{Case 2. $1<p<d/\alpha$:}\quad Assume that $f^p\mu^{p/q}\nu\in L^1(\mathbb{R}^d)$. We first use the second conclusion of Lemma \ref{maximal pointwise domination} with $f_0=f\nu\mu^{1/q}\nu^{-1/p'}$, $\omega_0=\mu^{1/q}\nu^{-1/p'}$, it yields that
\begin{align*}
M_\alpha^{\rho,\theta}(f\nu)&=M_\alpha^{\rho,\theta}(f\nu\mu^{1/q}\nu^{-1/p'}\mu^{-1/q}\nu^{1/p'})\\
&\leq M^{\rho,\theta}(f^p\nu^p\mu^{p/q}\nu^{-p/p'}\mu^{-1}\nu^{q/p'})
\Big(\int_{\mathbb{R}^d}f^p\nu^p\mu^{p/q}\nu^{-p/p'}\Big)^{\alpha/d}.
\end{align*}
Finally, according to Lemma \ref{schrodinger maximal mix weak type} with $f_1=f_0^p\omega_0^{-q}\nu^{-q}$, $\mu_1=\mu\nu^{-q/p'}$ and $\nu_1=\nu^q$, we obtain
\begin{align*}
&\mu\nu^{q/p}\Big(\Big\{x\in\mathbb{R}^d:\frac{M_{\alpha}^{\rho,\theta}(f\nu)(x)}{\nu(x)}>t\Big\}\Big)
^{p/q}\\
&\quad\leq\mu\nu^{q/p}\Big(\Big\{x\in\mathbb{R}^d:
\frac{M^{\rho,\theta}(f_0^p\omega_0^{-q})(x)}{\nu(x)^q}>\frac{t^q}
{\|f_0^p\|_{L^1(\mathbb{R}^d)}^{\alpha q/d}}\Big\}\Big)^{p/q}\\
&\quad=\mu\nu^{q/p}\Big(\Big\{x\in\mathbb{R}^d:
\frac{M^{\rho,\theta}(f_0^p\omega_0^{-q}\nu^{-q}\nu^q)(x)}{\nu(x)^q}>\frac{t^q}
{\|f_0^p\|_{L^1(\mathbb{R}^d)}^{\alpha q/d}}\Big\}\Big)^{p/q}\\
&\quad\lesssim t^{-p}\Big(\int_{\mathbb{R}^d}f_0^p\Big)^{p\alpha/d}
\Big(\int_{\mathbb{R}^d}f_0^p\omega_0^{-q}\nu^{-q}\mu\nu^{q/p}\Big)^{p/q}\\
&\quad=t^{-p}\int_{\mathbb{R}^d}f^p\mu^{p/q}\nu.
\end{align*}
\end{proof}

Now, we devoted to proving Theorem \ref{mix schrodinger fractional}. We use a very simple technique to avoid using the extrapolation theorem and Coifman type inequality as in \cite{BeCaPr}.
\begin{proof}[Proof of Theorem \ref{mix schrodinger fractional}]
We use notations in the proof of Theorem \ref{schrodinger fractional weak type}. We have proved that
\begin{align*}
|\mathcal{L}^{-\alpha/2}f(x)|\leq I_\alpha(|f_1|)(x)+|\mathcal{L}^{-\alpha/2}f_2(x)|,~x\in\mathbb{R}^d.
\end{align*}
It follows that
\begin{align}\label{fraction zuizong}
&\mu\nu^{q/p}\Big(\Big\{x\in\mathbb{R}^d:\frac{|\mathcal{L}^{-\alpha/2}(f\nu)(x)|}{\nu(x)}>t\Big\}\Big)
^{p/q}\\
&\quad\leq\mu\nu^{q/p}\Big(\Big\{x\in \mathbb{R}^d:\frac{|\mathcal{L}^{-\alpha/2}(f_2\nu)(x)|}{\nu(x)}>t/2\Big\}\Big)
^{p/q}\nonumber\\
&\qquad+\sum_{j\in\mathbb{N}}\mu\nu^{q/p}\Big(\Big\{x\in B_j:\frac{I_\alpha(|f_1|\nu)(x)}{\nu(x)}>t/2\Big\}\Big)
^{p/q}.\nonumber
\end{align}

Using \eqref{global dominate} and Theorem \ref{mix schrodinger fractional maximal}, we obtain
\begin{align}\label{fractional fen1}
&\mu\nu^{q/p}\Big(\Big\{x\in \mathbb{R}^d:\frac{|\mathcal{L}^{-\alpha/2}(f_2\nu)(x)|}{\nu(x)}>t/2\Big\}\Big)^{p/q}\\
&\quad\lesssim\mu\nu^{q/p}\Big(\Big\{x\in \mathbb{R}^d:\frac{M_\alpha^{\rho,\theta}(f_2\nu)(x)}{\nu(x)}>t/2\Big\}\Big)
^{p/q}\nonumber\\
&\quad\lesssim\frac{1}{t^p}\int_{\mathbb{R}^d}|f(x)|^p\mu(x)^{p/q}\nu(x)dx.\nonumber
\end{align}

On the other hand, since $A_{s}^{\rho,\theta}(\mathbb{R}^d)\subset A_{s}^{\rho,loc}(\mathbb{R}^d)$, by the assumption of $\nu$, we have $\nu^q\in A_{s}^{\rho,loc}(\mu\nu^{-q/p'})$ for some $s>1$. The same reasoning as \eqref{local claim}, we get $\nu^q|_{\widetilde{B_j}}\in A_s({\widetilde{B_j},\mu\nu^{-q/p'}})$. Then $\nu^q|_{\widetilde{B_j}}$ admits an extension $\nu_j^q$ which satisfies $\nu_j^q\in A_s(\mu\nu^{-q/p'})\subset A_\infty(\mu\nu^{-q/p'})$. Denote $\omega=\mu\nu^{-q/p'}$. Similarly, $\omega|_{\widetilde{B_j}}$ admits an extension $\omega_j$ which satisfies $\omega_j\in A_1(\mathbb{R}^d)$. Therefore, by Theorem \ref{mix fractional},
\begin{align*}\label{}
&\mu\nu^{q/p}\Big(\Big\{x\in B_j:\frac{I_\alpha(|f_1|\nu)(x)}{\nu(x)}>t\Big\}\Big)^{p/q}\\
&\quad=\mu\nu^{-q/p'}\nu^q\Big(\Big\{x\in B_j:\frac{I_\alpha(|f_1|\nu)(x)}{\nu(x)}>t\Big\}\Big)^{p/q}\\
&\quad=\omega_j\nu_j^q\Big(\Big\{x\in B_j:\frac{I_\alpha(|f_1|\nu)(x)}{\nu(x)}>t\Big\}\Big)^{p/q}\\
&\quad\leq\omega_j\nu_j^q\Big(\Big\{x\in \mathbb{R}^d:\frac{I_\alpha(|f_1|\nu)(x)}{\nu(x)}>t\Big\}\Big)^{p/q}\\
&\quad\lesssim t^{-p}\int_{\mathbb{R}^d}|f(x)|\chi_{\widetilde{B_j}}(x)(\omega_j(x)\nu_j(x)^q)^{p/q}dx\\
&\quad=t^{-p}\int_{\mathbb{R}^d}|f(x)|\chi_{\widetilde{B_j}}(x)\mu^{p/q}(x)\nu(x)dx.
\end{align*}
From this and Lemma \ref{covering lemma}, we have
\begin{align*}
\sum_{j\in\mathbb{N}}\mu\nu^{q/p}\Big(\Big\{x\in B_j:\frac{I_\alpha(|f_1|\nu)(x)}{\nu(x)}>t\Big\}\Big)^{p/q}\lesssim
t^{-p}\int_{\mathbb{R}^d}|f(x)|\mu^{p/q}(x)\nu(x)dx.
\end{align*}
This, together with \eqref{fraction zuizong} and \eqref{fractional fen1}, allows us to get the desired result.
\end{proof}


\end{document}